\documentclass[a4paper,11pt,oneside]{article} 
\usepackage{amsfonts,amssymb}
\usepackage{color} \usepackage{mathrsfs} \let\mathcal\mathscr
\usepackage[all,ps,cmtip]{xy}

\usepackage[utf8]{inputenc}
\usepackage{fancyhdr}
\usepackage[T1]{fontenc}
\usepackage{graphicx}
\usepackage{tikz}
\usepackage[left=3cm,right=3cm,top=4cm,bottom=4cm]{geometry}
\usepackage{amsmath,amsfonts,amssymb}
\usepackage{eurosym}
\usepackage{skull}
\usepackage{mathenv}
\usepackage{pstricks}

\title{\sc Scheme-theoretic Whitney conditions}
\author{\sc Roland Abuaf \footnote{E-mail :\textit{rabuaf@gmail.com. \textcolor{white}{Rectorat de Paris, département de mathématiques, 47 rue des \'Ecoles, 75005, Paris, France}}}}

\usepackage{amsfonts,amssymb}
\usepackage{color} \usepackage{mathrsfs} \let\mathcal\mathscr

\usepackage[all,ps,cmtip]{xy}

\DeclareGraphicsExtensions{eps}
\DeclareGraphicsExtensions{pdf}

\usepackage{amsmath}
\usepackage{latexsym}

\newtheorem{theo}{Theorem}[subsection]

\newtheorem{exem}[theo]{Example}
\newtheorem{rem}[theo]{Remark}
\newtheorem{prop}[theo]{Proposition}
\newtheorem{prob}[theo]{Problem}
\newtheorem{quest}[theo]{Question}

\newtheorem{defi}[theo]{Definition}

\newtheorem{cor}[theo]{Corollary}
\newtheorem{conj}[theo]{Conjecture}

\def\OO{\mathcal{O}}

\def\R0{\mathrm{R^{0}}}

\def\OO{\mathcal{O}}

\def\hX{\widehat{X}}

\everymath{\displaystyle}

\newcommand{\eq}[1][r]
{\ar@<-3pt>@{-}[#1]
\ar@<-1pt>@{}[#1]|<{}="gauche"
\ar@<+0pt>@{}[#1]|-{}="milieu"
\ar@<+1pt>@{}[#1]|>{}="droite"
\ar@/^2pt/@{-}"gauche";"milieu"
\ar@/_2pt/@{-}"milieu";"droite"}

\newcommand{\incl}[1][r]
  {\ar@<-0.2pc>@{^(-}[#1] \ar@<+0.2pc>@{-}[#1]}

\newenvironment{proof}
{
\noindent
\textit{\underline{Proof}} :\\
$\blacktriangleright\;$%
}
{\hspace{\stretch{1}}%
$\blacktriangleleft$}

\begin{document}

\maketitle

\begin{abstract}
We investigate a scheme-theoretic variant of Whitney condition $(a)$. If $X$ is  a projective variety over the field of complex numbers and $Y \subset X$ a subvariety, then $X$ satisfies generically the scheme-theoretic Whitney condition $(a)$ along $Y$ provided that the projective dual of $X$ is smooth. We give applications to tangency of projective varieties over $\mathbb{C}$ and to convex real algebraic geometry. In particular, we prove a Bertini-type theorem for osculating plane of smooth complex space curves and a generalization of a Theorem of Ranestad and Sturmfels describing the algebraic boundary of an affine compact real variety. 
\end{abstract}

\vspace{\stretch{1}}

\newpage

\tableofcontents

\begin{section}{Introduction}

Let $X \subset \mathbb{R}^n$ be a non-degenerate compact convex body, whose interior contains $0$. Denote  by $X^* \subset (\mathbb{R}^n)^*$ the dual body. Let $x \in \partial X$, the set $x^{\perp} \cap X^*$ is called the \textbf{exposed face} of $X^*$ with respect to $x^{\perp}$. A point $x \in \partial X$ is said to be an $r$-\textbf{singular point} if the exposed face of $X^*$ with respect to $x^{\perp}$ is of dimension at least $r$. In\cite{AK}, Anderson and Klee give a sharp bound on the Hausdorff dimension of the set of $r$-singular points of a convex body. Namely, we have the:

\begin{theo}[\cite{AK}] \label{AK-theo} Let $X \subset \mathbb{R}^n$ be a non-degenerate compact convex body, whose interior contains $0$. The set of $r$-singular points of $X$ has Hausdorff dimension at most $n-r-1$.

\end{theo} 

Following work of Ranestad and Sturmfels \cite{RS}, there have been some interest to know whether a version of Theorem \ref{AK-theo} could possibly hold in complex algebraic geometry. Namely, I conjectured in \cite{abuaf1} the following:

\begin{conj}[\cite{abuaf1}] \label{conj-moi}
Let $X \subset \mathbb{P}^N$ be a smooth complex projective variety. Denote by $X^*_r$ the variety:

\begin{equation*}
X^*_r := \{ H^{\perp} \in X^*, \, \textrm{such that} \, \dim \langle (H \cap X)_{tan} \rangle \geq r \}.
\end{equation*}
Then, we have $\dim X^*_r \leq N-r-1$.
\end{conj}
In the above conjecture, $(H \cap X)_{tan}$ is the tangency scheme of $H$ along $X$ and $\langle (H \cap X)_{tan} \rangle$ is the linear span of this tangency scheme. This conjecture is wrong as shows the following example which was communicated to me by Voisin:

\begin{exem} \label{exemain}
Let $X = v_d(\mathbb{P}^n) \subset \mathbb{P}(S^d \mathbb{C}^{n+1}) = \mathbb{P}^{N}$, with $N = {n+d \choose d} -1$. Hyperplane sections of $X$ are isomorphic to hypersurfaces of degree $d$ in $\mathbb{P}^n$. Hence, the variety $X^* \subset (\mathbb{P}^{N})^*$ parametrizes singular hypersurfaces of degree $d$ in $\mathbb{P}^n$ (the variety $X^*$ is known as the discriminant). We focus on singular hypersurfaces which are cones over hypersurfaces of degree $d$ in $\mathbb{P}^{n-1}$. A trivial count of dimension shows that there is a  ${d+n-1 \choose d} + n-1$-dimensional family of such cones (choose a hypersurface of degree $d$ in $\mathbb{P}^{n-1}$ and then choose the vertex of the cone in $\mathbb{P}^n$.)

\bigskip

We denote by $X^*_{cone}$ the closure of the subset of $X^*$ which consists of points $H^{\perp} \in X^*$ such that $H \cap X$ is a cone over a hypersurface of degree $d$ in $\mathbb{P}^{n-1}$. Let $H^{\perp}$ be generic in $X^*_{cone}$. We will give a lower bound on the linear span of $(H \cap X)_{tan}$. We fix coordinates (say $x_0, \cdots, x_n$) on $\mathbb{P}^n$, such that the vertex of $H \cap X$ is the point $[1:0: \cdots :0]$. We denote by $F_H$ the degree $d$ homogeneous polynomial in the $x_i$ corresponding to $H \cap X$. Since $H \cap X$ is a cone with vertex $[1:0: \cdots :0]$, $F_H$ only depends from $x_1, \cdots, x_n$. The tangency locus $(H \cap X)_{tan}$ is defined in $X$ by the vanishing of the partial derivatives:

\begin{equation*}
\left\{ \dfrac{\partial F_H}{\partial x_1}, \cdots, \dfrac{\partial F_H}{\partial x_n} \right\}.
\end{equation*} 
(we do a slight abuse of notations here as we consider the scheme defined by the ideal $\left( \dfrac{\partial F_H}{\partial x_1}, \cdots, \dfrac{\partial F_H}{\partial x_n} \right)$ as a subscheme of $X$ and not a subscheme of $\mathbb{P}^n$).

Furthermore, as $x_0, \cdots, x_n$ is a system of coordinates for $\mathbb{P}^n$, we know that the ideals $\left(\dfrac{\partial F_H}{\partial x_i}\right)$ and $\left(x_0\dfrac{\partial F_H}{\partial x_i}, \cdots, x_n\dfrac{\partial F_H}{\partial x_i} \right)$ are equal for all $i$. Hence, the subscheme of $X$ defined by the ideal $\left(\dfrac{\partial F_H}{\partial x_1}, \cdots, \dfrac{\partial F_H}{x_n}\right)$ is equal to the subscheme of $X$ defined by the ideal $\left(x_i\dfrac{\partial F_H}{\partial x_j}\right)_{i \in \{0 \cdots n \}, j \in \{1, \cdots n \} }$. We deduce that the linear span of $(H \cap X)_{tan}$ in $\mathbb{P}^N$ is defined by the equations:

\begin{equation*}
x_j \dfrac{\partial F_H}{\partial x_i} =0,
\end{equation*}
for all $i \in \{1, \cdots, n \}$ and all $j \in \{0, \cdots n \}$. As a consequence, we find: \[\dim \langle (H \cap X)_{tan} \rangle \geq N - n(n+1).\] We have seen that $\dim X^*_{cone} = {d+n-1 \choose d} + n-1$. But ${d+n-1 \choose d} + n-1 > N- (N - n(n+1))-1 = n(n+1)-1$ as soon as $d \geq 4$ if $n \geq 3$. So that $X^*_{cone}$ is a counter-example to Conjecture \ref{conj-moi} as soon as $d \geq 4$ and $n \geq 3$. 

\end{exem}

For the applications found in \cite{RS}, only the set-theoretic version Conjecture \ref{conj-moi} is needed. It is stated in \cite{abuaf1} that this set-theoretic version is true and we have:

\begin{theo} \label{app1}
Let $X \subset \mathbb{P}^N$ be a projective variety. Denote by $\tilde{X}^*_r$ the variety:

\begin{equation*}
\tilde{X}^*_r := \{ H^{\perp} \in X^*, \, \textrm{such that} \, \dim \langle (H \cap X)^{red}_{tan} \rangle \geq r \}.
\end{equation*}
Then, we have $\dim \tilde{X}^*_r \leq N-r-1$.
\end{theo}
Here $\langle (H  \cap X)^{red}_{tan} \rangle$ is the linear span of the reduced tangency locus of $H$ with $X$. A complicated proof of this Theorem was announced in \cite{abuaf1}. In section 2 of this paper, we will give a slick proof of Theorem \ref{app1} based on the fact that Whitney condition $(a)$ is stratifying.

\bigskip

It was however noticed by Peskine and Zak that the Conjecture \ref{conj-moi} is intriguing precisely because it would imply the existence of unexpected bounds on the dimensions of families of osculating hyperplanes. With this perspective in mind, Theorem \ref{app1} seems a bit less appealing as it says nothing about osculating hyperplanes. In fact, in order to prove a result which would have a similar flavor as Conjecture \ref{conj-moi}, one has to replace the linear span of the tangency locus by the linear span of the union of the tangent spaces to the tangency locus. The difference between these two objects might look quite thin at first sight. Indeed, both objects coincide if the tangency locus is reduced : this is one of the reasons which explain why the erroneous conjecture \ref{conj-moi} was made in the first place. The example of the $d$-th Veronese embedding of $\mathbb{P}^n$ in $\mathbb{P}(S^d \mathbb{C}^{n+1})$ shows that these two objects may differ in general.

\begin{exem} \label{exem-bis}
Let $X = v_d(\mathbb{P}^n) \subset \mathbb{P}(S^d \mathbb{C}^{n+1}) = \mathbb{P}^{N}$. Let $H^{\perp} \in X^*$ such that $H \cap X$ is a cone over a smooth hypersurface in $\mathbb{P}^{n-1}$, with vertex denoted by $x$. The hyperplane $H$ is tangent to $X$ only at $x$. Furthermore, as explained in example \ref{exemain}, we have $\dim \langle (H \cap X)_{tan} \rangle \geq \binom{n+d}{d} -1 - n(n+1)$. On the other hand, we have $T_{(H \cap X)_{tan},x} \subset T_{X,x}$, so that $\dim \langle T_{(H \cap X)_{tan}} \rangle \leq n$. The quantities $\binom{n+d}{d} -1 - n(n+1)$ and $n$ differ as soon as $d \geq 4$ and $n \geq 3$.
\end{exem}

The linear span of the union of the tangent spaces to the tangency locus seems to be the correct linear span to consider as soon as stratification of the projective dual is concerned. Indeed, the main result of this paper is the:

\begin{theo} \label{main}
Let $X \subset \mathbb{P}^N$ be a smooth projective variety. Denote by $\overline{X}^*_r$ the variety:

\begin{equation*}
\overline{X}^*_r := \{ H^{\perp} \in X^*, \, \textrm{such that} \, \dim \langle T_{(H \cap X)_{tan}} \rangle \geq r \}.
\end{equation*}
Then, we have $\dim \overline{X}^*_r \leq N-r-1$.
\end{theo}
Here $\langle T_{(H \cap X)_{tan}} \rangle$ is the linear span of the union of the tangent spaces to the tangency scheme of $H$ with $X$. The proof of this result is based on the stratifying property of a certain scheme-theoretic version of Whitney condition $(a)$ for the pair $(X^*,Y)$ ($Y$ a subvariety of $X^*$), provided that $X$ is smooth. The smoothness assumption can be dropped if one defines $\langle T_{(H \cap X)_{tan}} \rangle  $ as the linear span of the Zariski closure of the union of the tangent spaces to the tangency scheme of $H$ with $X$ at smooth points of $X$. We will come back to this more technical statement in section $2$ of the paper.

\bigskip

In section $3$, we will give applications of this result to both the study of tangency of projective varieties defined over the field of complex numbers and to convex real algebraic geometry. In particular, we will prove a Bertini-type result for osculating planes of smooth complex space curves (see Theorem \ref{osc}) and a generalization of a result due to Ranestad and Sturmfels describing the algebraic boundary of an affine compact real variety (see Theorem \ref{RSmieux}).

\bigskip

\textbf{Acknowledgments} \, \, I am very grateful to Christian Peskine and Fyodor Zak for the numerous and fruitful discussions we have had on Conjecture \ref{conj-moi}. It was their insight that the the linear span of the union of the tangent spaces to the tangency scheme should replace the linear span of the tangency scheme in the statement of Conjecture \ref{conj-moi}. I am also indebted to Claire Voisin for sharing with me her counter-example to Conjecture \ref{conj-moi}. 

\end{section}

\begin{section}{Variations on the Whitney condition $(a)$}

\begin{subsection}{Whitney condition $(a)$}
In this section we gather some elementary facts about Whitney condition $(a)$ and we use them to prove the set-theoretic version of Conjecture \ref{conj-moi}. Recall the definition of Whitney condition $(a)$ (see \cite{Whitney} or \cite{Teissier} for instance)

\begin{defi} \label{whitneya}
Let $X \subset \mathbb{C}^n$ be an (algebraic) variety, $Y \subset X$ a subvariety of $X$ and $y \in Y$. We say that \textbf{the pair $(X,Y)$ satisfies the Whitney condition $(a)$ at $y$} if for any sequence $\{x_n\}$ of points of $X_{smooth}$ converging to $y$ such that the sequence of tangent spaces $ T_{X,x_n}$ converges (in the appropriate Grassmannian), the limit tangent space $\lim T_{X,x_n}$ contains $T_{Y,y}$.
\end{defi}

With notations as in definition \ref{whitneya}, the Whitney condition $(a)$ is easily seen to be equivalent to the following : for any sequence $\{x_n\}$ of points of $X_{smooth}$ converging to $y$ and any converging sequence of hyperplanes $\{H_n \}$ which are tangent to $X$ at $x_n$, the limiting hyperplane $\lim H_n$ contains $T_{Y,y}$. This definition may then be further reformulated as follows:

\begin{prop}
Let $X \subset \mathbb{C}^{N+1}$ be an algebraic variety and $Y \subset X$ a subvariety with a marked point $y \in Y$. Denote by $I_X \subset \mathbb{C}^{N+1} \times (\mathbb{C}^{N+1})^*$ the Zariski closure of the total space of $N^*_{X_{smooth}/\mathbb{C}^{N+1}}$ over $X_{smooth}$. We denote by

\begin{equation*}
\xymatrix{ & & \ar[lldd]_{p}  I_X \ar[rrdd]^{q}& &  \\
& & & & \\
(\mathbb{C}^{N+1})^* & & & & \mathbb{C}^{N+1}}
\end{equation*}
the canonical projections. The pair $(X,Y)$ satisfies the Whitney condition $(a)$ at $y$ if and only if $\{q(p^{-1}(y)) \}_{red} \subset T_{Y,y}^{\perp}$.
\end{prop}
Here $\{q(p^{-1}(y)) \}_{red}$ is the reduced scheme underlying $q(p^{-1}(y))$.

\begin{proof}
By definition of $I_X$, the reduced scheme $\{q(p^{-1}(y)) \}_{red}$ is the set of limits of converging sequences of hyperplanes tangent to $X_{smooth}$ at sequences of points which converge to $y$. Hence, the pair $(X,Y)$ satisfies the Whitney condition $(a)$ at $y$ if and only if $ \{q(p^{-1}(y)) \}_{red} \subset T_{Y,y}^{\perp}$.
\end{proof}
\bigskip

The Whitney condition $(a)$ is an interesting condition because it is stratifying. More precisely, we have:

\begin{theo}[\cite{Whitney}] \label{whiwhi1} Let $X \subset \mathbb{C}^{N+1}$ be an algebraic variety and let $Y \subset X$ be a subvariety. There exists a non-empty Zariski open subset $U \subset Y$ such that the pair $(X,Y)$ satisfies the Whitney condition $(a)$ for all $y \in U$.
\end{theo}
This result was originally proved by Whitney \cite{Whitney}. We refer to \cite{Teissier} for a quick proof in the algebraic setting. As a consequence of this stratifying property, one can immediately deduce the following:

\begin{theo} \label{app11}
Let $X \subset \mathbb{P}^N$ be a projective variety. Denote by $\tilde{X}^*_r$ the variety:

\begin{equation*}
\tilde{X}^*_r := \{ H^{\perp} \in X^*, \, \textrm{such that} \, \dim \langle (H \cap X)^{red}_{tan} \rangle \geq r \}.
\end{equation*}
Then, we have $\dim \tilde{X}^*_r \leq N-r-1$.
\end{theo}

\begin{proof}
Let $\hX \subset \mathbb{C}^{N+1}$ be the affine cone over $X$. Let $I_{\hX} \subset \mathbb{C}^{N+1} \times (\mathbb{C}^{N+1})^*$ be the Zariski closure of the total space of $N^*_{\hX_{smooth}/\mathbb{C}^{N+1}}$ over $\hX_{smooth}$. We denote by $\widehat{X^*}$ the image of the projection of $I_X$ on $(\mathbb{C}^{N+1})^*$ (note that $ \widehat{X^*}$ is the affine cone over the projective dual of $X$). If $I_{\widehat{X^*}} \subset \mathbb{C}^{N+1} \times (\mathbb{C}^{N+1})^*$ is the Zariski closure of the total space of $N^*_{\widehat{X^*}_{smooth}/(\mathbb{C}^{N+1})^*}$ over $\widehat{X^*}_{smooth}$, the Lagrangian duality (also called \textit{Reflexivity}) asserts that:

\begin{equation*}
I_{\hX} = I_{\widehat{X^*}}.
\end{equation*}

Denote by :
\begin{equation*}
\xymatrix{ & & \ar[lldd]_{p}  I_X \ar[rrdd]^{q}& &  \\
& & & & \\
(\mathbb{C}^{N+1})^* & & & & \mathbb{C}^{N+1}}
\end{equation*}
 the canonical projections. Let $Z_r$ be the affine cone over $\tilde{X}^*_r$. By definition, 
\begin{equation*}
Z_r = \{ H^{\perp} \in \widehat{X^*}, \, \textrm{such that} \,   \dim \langle \{q(p^{-1}(H)) \}_{red} \rangle \geq r+1 \}.
\end{equation*}

Let $H^{\perp}$ be a generic point in $Z_r$. By Theorem \ref{whiwhi1}, we know that $\langle \{q(p^{-1}(H^{\perp})) \}_{red} \subset T_{Z_r,H^{\perp}}^{\perp}$. This implies that $\dim Z_r \leq N+1-(r+1)$. As a consequence, we conclude that $\dim \tilde{X}^*_r \leq N-1-r$.
\end{proof}
\bigskip

\end{subsection}

\begin{subsection}{Scheme-theoretic Whitney condition $(a)$}

In this section we will explore a scheme-theoretic variant of Whitney condition $(a)$ for varieties which are afinne cones over projective varieties. We prove that this condition is stratifying if the projective dual of the ambient variety is smooth. In the following, if $Z \subset \mathbb{C}^{N+1}$ is the affine cone over a closed projective scheme say $\mathbb{P}(Z) \subset \mathbb{P}^{N}$, we denote by $Z^{\bullet}$ the scheme $Z \backslash \{0\}$ and we let $\Big\langle \overline{T_{Z^{\bullet}}} \Big\rangle$ be the linear span of the Zariski closure of the the union of the tangent spaces to $Z^{\bullet}$. Notice that we have the chain of inclusions:
\[ \langle Z_{red} \rangle \subset \Big\langle \overline{T_{Z^{\bullet}}} \Big\rangle \subset \langle Z \rangle\]
and these inclusions are all equalities if $Z$ is reduced. Example \ref{exem-bis} shows that they can be both strict if $Z$ is non reduced. More genereally, if $Z \subset X \subset \mathbb{C}^{N+1}$ is a closed subscheme of $X$ in $\mathbb{C}^{N+1}$, we denote by $\Big\langle \overline{ T_{Z, X_{smooth}}} \Big\rangle$, the linear span in $\mathbb{C}^{N+1}$ of the Zariski closure of the union of the tangent spaces to $Z$ at smooth points of $X$.

\begin{defi}
Let $X \subset \mathbb{C}^{N+1}$ be the affine cone over a projective variety and let $Y \subset X$ be the affine cone over a projective subvariety of $\mathbb{P}(X)$. Let $y \in Y$ and let $I_X \subset \mathbb{C}^{N+1} \times (\mathbb{C}^{N+1})^*$ be the closure of $N^*_{X_{smooth}/\mathbb{C}^{N+1}}$ in $\mathbb{C}^{N+1} \times (\mathbb{C}^{N+1})^*$. We denote by 

\begin{equation*}
\xymatrix{ & & \ar[lldd]_{p}  I_X \ar[rrdd]^{q}& &  \\
& & & & \\
(\mathbb{C}^{N+1})^* & & & & \mathbb{C}^{N+1}}
\end{equation*}
 the canonical projections. We say that the pair $(X,Y)$ satisfies the \textbf{scheme-theoretic Whitney condition $(a)$ at $y$} if $\Big\langle \overline{T_{p(q^{-1}(y))^{\bullet}}} \Big\rangle \subset T_{Y,y}^{\perp}$.

\end{defi}

\begin{rem} \label{remrem}
\begin{enumerate}
\item It is obvious that the pair $(X,Y)$ satifies the Whitney condition (a) at $y$ if it satisfies the scheme-theoretic Whitney conditions (a) at $y$. The converse is false.

\item Since $X$ is the affine cone over a projective variety (denoted by $\mathbb{P}(X)$), the variety $X^* = p(I_X)$ is the affine cone over the projective dual of $\mathbb{P}(X)$, which we denote by $\mathbb{P}(X^*)$. For all $y \in Y$, with $y \neq 0$, the scheme $p(q^{-1}(y))$ is the affine cone over a projective scheme (which we denote by $\mathbb{P} \{p(q^{-1}(y)) \}$). One readily checks that $\Big\langle \overline{T_{p(q^{-1}(y)), X^*_{smooth}}} \Big\rangle$ is the affine cone over the linear subspace of $(\mathbb{P}^N)^*$ spanned by the union of the (embedded in $(\mathbb{P}^N)^*$) tangent spaces to $\mathbb{P} \{p(q^{-1}(y)) \}$ at smooth points of $\mathbb{P}(X^*)$. In particular, if $\mathbb{P} \{p(q^{-1}(y)) \} \subset \mathbb{P}( X^*)_{smooth} $ , then $\Big\langle \overline{T_{p(q^{-1}(y)), X^*_{smooth}}} \Big\rangle = \Big\langle \overline{T_{p(q^{-1}(y))^{\bullet}}} \Big\rangle$ is the affine cone in $(\mathbb{C}^{N+1})^*$ over the linear subspace of $(\mathbb{P}^N)^*$ spanned by the union of the (embedded) tangent spaces to $\mathbb{P} \{p(q^{-1}(y)) \}$.
\end{enumerate}
\end{rem}
\bigskip

Now we come to the main result of this section. 

\begin{theo} \label{mainmain}
Let $X \subset \mathbb{C}^{N+1}$ be the affine cone over a projective variety and let $Y \subset X$ be the affine cone over a projective subvariety of $\mathbb{P}(X)$. We denote by :
\begin{equation*}
\xymatrix{ & & \ar[lldd]_{p}  I_X \ar[rrdd]^{q}& &  \\
& & & & \\
(\mathbb{C}^{N+1})^* & & & & \mathbb{C}^{N+1}}
\end{equation*}
the closure of $N^*_{X_{smooth}/\mathbb{C}^{N+1}}$ in $\mathbb{C}^{N+1} \times (\mathbb{C}^{N+1})^*$ with its canonical projections. Assume that for generic $y \in Y$, the projectivization $\mathbb{P} \{ p(q^{-1}(y)) \}$ lies in $\mathbb{P} (X^*)_{smooth}$. Then, the pair $(X,Y)$ satisfies generically along $Y$ the scheme-theoretic Whitney condition $(a)$.
\end{theo}

\begin{proof}

We know that for generic $y \in Y$, the projectivization $\mathbb{P} \{ p(q^{-1}(y)) \}$ lies in $\mathbb{P}(X^*)_{smooth}$. Hence, there exists a dense open subset $U \subset$Y, such that for all $y \in U$, the projectivized fiber $\mathbb{P} \{ p(q^{-1}(y)) \}$ lies in $\mathbb{P}(X^*)_{smooth}$. Put differently, for all $ y \in U$ and for all $z \in p(q^{-1}(y))$ such that $z \neq 0$, we have $z \in X^*_{smooth}$.

\bigskip

By Lagrangian duality (also called reflexivity, see \cite{GKZ}, chapter 1), we know that $I_X = I_X^*$ is a Lagragngian subvariety of $\mathbb{C}^{N+1} \times \left(\mathbb{C}^{N+1} \right)^*$. Hence, for all $(x,z) \in I_X$ such that $z \in X^*_{smooth}$, the vector space $\Omega_{I_X, (x,z)}$ is a Lagrangian subspace of $\mathbb{C}^{N+1} \times (\mathbb{C}^{N+1})^*$ with respect to the natural symplectic form (denoted by $\omega$) on $\mathbb{C}^{N+1} \times (\mathbb{C}^{N+1})^*$. In particular, for $y \in U$, the sheaf $\Omega_{I_X, q^{-1}(y)^{\bullet}}$ is a Lagrangian subbundle of $\mathbb{C}^{N+1} \times (\mathbb{C}^{N+1})^* \otimes \OO_{q^{-1}(y)^{\bullet}}$ with respect to $\omega$. As a consequence, the two form:

\begin{equation} \label{eq1}
\omega : \Omega_{I_X, q^{-1}(y)^{\bullet}} \times \Omega_{I_X, q^{-1}(y)^{\bullet}} \longrightarrow \OO_{q^{-1}(y)^{\bullet}}
\end{equation}
is everywhere vanishing.

\bigskip

We consider the restriction of the projection map $q : q^{-1}(Y) \longrightarrow Y$. Since we are working over $\mathbb{C}$, up to shrinking $U$, one can assume that the codifferential map gives an injection $0 \longrightarrow q^*\Omega_{U} \longrightarrow \Omega_{q^{-1}(U)}$ (this is the generic smoothness Theorem). As a consequence, for all $y \in U$, we have an injection of sheaves:

\begin{equation} \label{eq2}  
0 \longrightarrow q^*\Omega_{Y,y} \longrightarrow \Omega_{q^{-1}(U), q^{-1}(y)}.
\end{equation}
Moreover, as $q^{-1}(U)$ is a subscheme of $I_{X}$, we have a surjection:

\begin{equation} \label{eq3}  
\Omega_{I_X, q^{-1}(y)} \longrightarrow \Omega_{q^{-1}(U), q^{-1}(y)} \longrightarrow 0.
\end{equation}

\bigskip

Notice that we also have an injection: 
\begin{equation} \label{eq4}
0 \longrightarrow p^* \Omega_{X^*, p(q^{-1}(y))^{\bullet}} \longrightarrow \Omega_{I_X, q^{-1}(y)^{\bullet}}.
\end{equation}
Indeed, the morphism $p : I_X \longrightarrow X^*$ is smooth over $p(q^{-1}(y))^{\bullet}$ as it identifies with the fibration $N^*_{X^*_{smooth}/(\mathbb{C}^{N+1})^*} \longrightarrow X^*_{smooth}$. Finally, as $p(q^{-1}(y))^{\bullet}$ is a subscheme of $X^*$, we have a surjection:

\begin{equation} \label{eq5}
\Omega_{X^*, p(q^{-1}(y))^{\bullet}} \longrightarrow \Omega_{p(q^{-1}(y))^{\bullet}} \longrightarrow 0.
\end{equation}

Combining equations \ref{eq1}, \ref{eq2}, \ref{eq3}, \ref{eq4} and \ref{eq5}, we find that the two form:

\begin{equation*}
\omega : q^* \Omega_{Y,y} \times p^* \Omega_{p(q^{-1}(y))^{\bullet}} \longrightarrow \OO_{p(q^{-1}(y))^{\bullet}}
\end{equation*}
is everywhere vanishing. Taking tensor products with $\mathbb{C}(z)$ for any $z \in p(q^{-1}(y))^{\bullet}$ and then the dual, we find that the two form:

\begin{equation*}
\omega : q^* T_{Y,y} \times p^* T_{p(q^{-1}(y))^{\bullet},z} \longrightarrow \mathbb{C}
\end{equation*}
is zero. This means that the vector space $T_{Y,y}$ is orthogonal to $T_{p(q^{-1}(y)),z}$ for all $z \in p(q^{-1}(y))^{\bullet}$. As a consequence, we get the inclusion

\begin{equation*}
\Big\langle \overline{T_{p(q^{-1}(y))^{\bullet}}} \Big\rangle \subset T_{Y,y}^{\perp}.
\end{equation*}
This concludes the proof.
\end{proof}
\bigskip

As a matter of fact, it is not difficult to see that one has the more general statement:

\begin{theo}[variant of Theorem \ref{mainmain}] \label{mainmain2}
Let $X \subset \mathbb{C}^{N+1}$ be the affine cone over a projective variety and let $Y \subset X$ be the affine cone over a projective subvariety of $\mathbb{P}(X)$. We denote by :
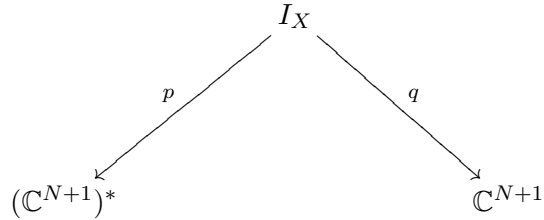
\begin{equation*}
\xymatrix{ & & \ar[lldd]_{p}  I_X \ar[rrdd]^{q}& &  \\
& & & & \\
(\mathbb{C}^{N+1})^* & & & & \mathbb{C}^{N+1}}
\end{equation*}
the closure of $N^*_{X_{smooth}/\mathbb{C}^{N+1}}$ in $\mathbb{C}^{N+1} \times (\mathbb{C}^{N+1})^*$ with its canonical projections. Then for generic $y \in Y$, we have:

\begin{equation*}
\Big\langle \overline{T_{p(q^{-1}(y)), X^*_{smooth}}} \Big\rangle \subset T_{Y,y}^{\perp}.
\end{equation*}
\end{theo}
The proof is exactly the same as for Theorem \ref{mainmain} and the details are left to the reader.

\begin{cor}
Let $X \subset \mathbb{P}^N$ be a smooth projective variety. Denote by $\overline{X}^*_r$ the variety:
\begin{equation*}
\overline{X}^*_r := \{ H^{\perp} \in X^*, \, \textrm{such that} \, \dim \langle T_{(H \cap X)_{tan}} \rangle \geq r \}.
\end{equation*}
Then, we have $\dim \overline{X}^*_r \leq N-r-1$.
\end{cor}

\begin{proof}
Apply Theorem \ref{mainmain} to the pair $(X^*, \overline{X}^*_r)$, knowing that $(X^*)^* = X$ is smooth.
\end{proof}

\bigskip

The proof of Theorem \ref{mainmain} shows that a variant of the scheme-theoretic Whitney condition $(a)$ (namely that $\langle \overline{T_{p(q^{-1}(y)), X^*_{smooth}}} \rangle \subset T_{Y,y}^{\perp}$) is always stratifying. On the other hand, the following example suggests that this variant is far from being optimal. 

\begin{exem}
Let $X \subset \mathbb{C}^5$ be the affine cone over a generic hyperplane section of $\mathbb{P}^1 \times \mathbb{P}^2 \subset \mathbb{P}^5$. The surface $\mathbb{P} (X) \subset \mathbb{P}^4$ is a ruled surface, for which the principal axis of the ruling is denoted by $\mathbb{P}(L)$. The projective dual of $\mathbb{P}(X)$ is a cubic hypersurface, whose singular locus is precisely $\mathbb{P}(L)^{\perp}$ (we refer to \cite{zak1} and \cite{abuaf0} for more details on this example). As $X$ is smooth outside $0$, the pair $(X,L)$ obviously satisfies the scheme-theoretic Whitney condition $(a)$ at all $z \in L$ with $z \neq 0$. On the other hand, with notations as in Theorem \ref{mainmain2}, we have $p(q^{-1}(z)) \subset X^*_{sing}$, for any $z \in L$. Hence $\langle \overline{T_{p(q^{-1}(z)), X^*_{smooth}}} \rangle = 0$ for any $z \in L$ and the conclusion of Theorem \ref{mainmain2} is void in that situation.
\end{exem}

This suggests that there is a some room for improving Theorem \ref{mainmain} and Theorem \ref{mainmain2}. In particular, one could ask the following question:

\begin{quest}
Let $X \subset \mathbb{C}^{N+1}$ be the affine cone over a projective variety and let $Y \subset X$ be the affine cone over a projective subvariety of $\mathbb{P}(X)$. Does the pair $(X,Y)$ necessarily satisfies the scheme-theoretic Whitney condition $(a)$ at the generic point of $Y$?
\end{quest}

\end{subsection}
\end{section}

\begin{section}{Applications of the scheme-theoretic Whitney conditions}
\begin{subsection}{Classification of projective surfaces with a maximal family of maximally tangent hyperplanes}
The goal of this section is to give an application of Theorem \ref{mainmain} to the classification of smooth projective varieties with large family of \textit{maximally tangent hyperplanes}. 

\begin{theo} \label{surface}
Let $X \subset \mathbb{P}^N$ be a smooth projective non-degenerate surface and $N \geq 5$. Denote by $X^*(1)$ the set of hyperplanes which are tangent to $X$ along a curve. Then 
\begin{equation*}
\dim X^*(1) \leq N-3.
\end{equation*}
If $\dim X^*(1) = N-3$, then the family of tangency curves in $X$ of hyperplanes parametrized by $X^*(1)$ is at least $1$-dimensional. If it is at least $2$-dimensional, then $X$ is the Veronese surface in $\mathbb{P}^5$.

\end{theo}

\begin{proof}
Let $H^{\perp} \in X^*(1)$ be generic. Then $\Big\langle T_{(H \cap X)_{tan}} \Big\rangle \geq 1$ and we have equality if and only if $(H \cap X)_{tan}$ is scheme-theoretically a line. By Theorem \ref{mainmain}, this implies that $\dim X^*(1) \leq N-2$. 

\bigskip

Assume that $\dim X^*(1) = N-2$. For any $H^{\perp} \in X^*(1)$ we denote by $L_H$ the tangency scheme of $H$ along $X$. This is a line. Fix a general $H_0^{\perp} \in X^*(1)$. Assume that the dimension of the set of $H^{\perp} \in X^*(1)$ which are tangent to $X$ along $L_{H_0}$ is at least $N-3$. Since $N \geq 5$, we would find a $\mathbb{P}^2 \subset \mathbb{P}^N$ which is tangent to $X$ along $L_{H_0}$. By Zak's Theorem on Tangency (\cite{zak0}) this is impossible. Hence, the dimension of the set of $H^{\perp} \in X^*(1)$ which are tangent along $L_{H_0}$ is at most $N-4$. This implies that the family of lines $\{L_H \}_{H^{\perp} \in X^*(1)}$ is at least $2$-dimensional. As a consequence, the variety $X$ is a $\mathbb{P}^2$ linearly embedded in $\mathbb{P}^N$. This is impossible by the non-degeneracy hypothesis. As a conclusion, we get that:

\begin{equation*}
\dim X^*(1) \leq N-3.
\end{equation*}

\bigskip

Assume that $\dim X^*(1) = N-3$. Let $I_X = \mathbb{P}(N^*_{X/\mathbb{P}^N}(1)) \subset \mathbb{P}^N \times (\mathbb{P}^N)^*$. We denote by
\begin{equation*}
\xymatrix{ & & \ar[lldd]_{p}  I_X \ar[rrdd]^{q}& &  \\
& & & & \\
(\mathbb{P}^{N})^* & & & & \mathbb{P}^{N}}
\end{equation*}
the canonical projections. Let us put $Z = q(p^{-1}(X^*(1)))$. Assume that $\dim Z = 1$. Then, there exists an irreducible component of $Z$ which is included in $(H \cap X)_{tan}$ for all $H^{\perp} \in X^*$. Since $\dim X^*(1) = N-3$ and $N \geq 5$, this would imply that there is a $\mathbb{P}^2$ which is tangent to $X$ along $Z$. By Zak's Theorem on Tangency, this is impossible. As a consequence, we have:

\begin{equation*}
p(q^{-1}(X^*(1))) = X,
\end{equation*}
so that the family of tangency curves in $X$ of hyperplanes parametrized by $X^*(1)$ is at least $1$-dimensional.

\bigskip

Assume that this family is at least $2$-dimensional. For any $H^{\perp} \in X^*$, we denote by $X_H$ the tangency locus of $H$ with $X$. Since the family $\{X_H\}_{H^{\perp} \in X^*(1)}$ is assumed to be at least $2$-dimensional, we find that for any $x,y \in X$, there exists a hyperplane $H^{\perp} \in X^*(1)$ such that $x$ and $y$ lie on the curve $X_H$. Since $\dim X^*(1) =  N-3$, Theorem \ref{app11} implies that for generic $H^{\perp} \in X^*(1)$, the linear span of $(X_H)_{red}$ is $2$-dimensional, which proves that $(X_H)_{red}$ is a plane curve. Assume that $\deg (X_H)_{red} \geq 3$. As for any $x,y \in X$, there exists $H^{\perp} \in X^*(1)$ such that $x$ and $y$ lie on the curve $(X_H)_{red}$, this shows that every bisecant to $X$ is at least a trisecant. This is impossible by the trisecant lemma. We deduce that for generic $H^{\perp} \in X^*(1)$, the curve $(X_H)_{red}$ is a plane conic. Note that for generic $H^{\perp} \in X^*(1)$, the conic $(X_H)_{red}$ can not be singular as $X$ does not contain a two dimensional family of lines (otherwise $X$ would be a $\mathbb{P}^2$ linearly embedded in $\mathbb{P}^5$). We infer that $X$ is covered by an at least two-dimensional family of conics, the generic member of this family of conics being smooth. We conclude that $X$ is the Veronese surface in $\mathbb{P}^5$.

\end{proof}

It would be interesting to know if one can get a classification result as in Theorem \ref{surface} when $\dim X^*(1) = N-3$ and the family of tangency loci along $X$ of hyperplanes parametrized by $X^*(1)$ is exactly $1$-dimensional. By Theorem \ref{mainmain}, this is equivalent to the following problem:

\begin{prob}
Classify all smooth surfaces $X \subset \mathbb{P}^N$ with $N \geq 5$ such that there is a $1$-dimensional family of $\mathbb{P}^3$ whose tangency loci with $X$ are curves.
\end{prob}
This classification must necessarily take into account ruled surfaces as many of them do satisfy the hypothesis of the problem.

\end{subsection}

\begin{subsection}{A Bertini-type Theorem for osculating planes of space curves}
In this section we will focus on a Bertini-type result for osculating hyperplanes that can be obtained as a consequence of Theorem \ref{mainmain}.

\begin{theo} \label{osc}
Let $C \subset \mathbb{P}^3$ be a smooth space curve. A generic osculating plane to $C$ is not tangent to $C$ outside its osculating locus. 
\end{theo}

\begin{proof}
We proceed by absurd. Denote by $C_{osc} \subset (\mathbb{P}^3)^*$ the curve of osculating planes to $C$. We know (see \cite{Piene}) that for generic $H^{\perp} \in C_{osc}$, the plane $H$ osculates $X$ to first order at exactly one point, which we denote by $c_H$. Assume that for generic $H^{\perp} \in C_{osc}$, the plane $H$ is tangent to $C$ outside $c_H$. We denote by $C_H^{extra}$ the reduced tangency locus of $H$ with $C$ located outside $c_H$. Since $\dim C_{osc} = 1$, Theorem \ref{mainmain} implies that for generic $H^{\perp} \in C_{osc}$, we have:

\begin{equation*}
\langle C_H^{extra} \rangle \subset  T_{(H \cap C)_{tan},c_H} ,
\end{equation*}
where $T_{(H \cap C)_{tan},c_H}$ is the tangent space to the tangency locus of $H$ along $C$ at $c_H$. As $H$ is osculating $C$ at $c_H$ to the first order, we know that $T_{(H \cap C)_{tan},c_H} = T_{C,c_H}$. We deduce that the tangent $T_{C,c_H}$ cuts $C$ outside of $c_H$.  This is true for generic $H^{\perp} \in X_{osc}$, so that for all $c \in C$, the  tangent $T_{C,c}$ cuts again $C$ outside of $c$. Since $C$ is smooth, we get a contradiction with Kaji's result on tangentially degenerate curve \cite{Kaji}.
\end{proof}

\bigskip

It is well known (and often used when dealing with plane projections of space curves) that a \textit{generic} curve in $\mathbb{P}^3$ has only finitely many osculating planes which are tangent to the curve outside the osculating locus \cite{BF}, \cite{GH} (chapter 2, section 5), \cite{Wall}. To the best of my knowledge, it was not known that all smooth curves in $\mathbb{P}^3$ are \textit{generic}, when it comes to the dimension of the family of doubly tangent osculating hyperplanes.

\begin{rem}
\begin{enumerate}
\item The reflexivity Theorem for osculating varieties \cite{Piene} shows that \textbf{any} space curve in $\mathbb{P}^3$ has a $1$-dimensional family of bitangent osculating hyperplanes if and only if the generic tangent line to the osculating curve is at least a trisecant. On the other hand, the proof of the main result of \cite{Kaji} only works for curves which are mildly singular (see \cite{pirola} for some improvements of Kaji's result). It would be interesting to know if one can get a direct proof of Theorem \ref{osc} which would work for singular curves. This would imply in particular that any space curve is tangentially non-degenerate.

\item One expects that similar results to Theorem \ref{osc} hold for higher dimensional varieties. Namely, under appropriate hypotheses on the dimension and codimension of $X \subset\mathbb{P}^N$, it is not inconsiderate to believe that a generic osculating hyperplane is not tangent to $X$ outside of its osculating locus.
\end{enumerate}
\end{rem}

\end{subsection}

\begin{subsection}{The algebraic boundary of an affine compact real variety}

Let $X_{\mathbb{R}} \subset \mathbb{R}^N$ be an affine compact real variety. We denote by $P := \mathrm{Conv}(X_{\mathbb{R}})$ the closed convex hull of $X$ in $\mathbb{R}^N$ and by $\partial_a P$ the Zariski closure in $\mathbb{P}^{N}_{\mathbb{C}}$ of the boundary of $\mathrm{Conv}(X_{\mathbb{R}})$. The complex projective hypersurface $\partial_a P$ is called the \textit{algebraic boundary} of $X_{\mathbb{R}}$ and has become recently an object of active study for both algebraic geometry and optimization theory (see \cite{RS, RS2, BHORS, ORSV, Sinn} for instance).

In \cite{RS}, a result describing the algebraic boundary of an affine compact real variety in terms of duals of some singular strata of the projective dual of its projective closure was proved under a restricting hypothesis: namely that only finitely many hyperplanes are tangent to the complex projective closure of the given variety at infinitely many points. Using Theorem \ref{app11}, we are able to get rid of the restricting hypothesis.

\begin{theo}\label{RSmieux}
Let $X_{\mathbb{R}} \subset \mathbb{R}^N$ be a smooth and compact real algebraic variety that affinely spans $\mathbb{R}^N$. We denote by $X \subset \mathbb{P}^N_{\mathbb{C}}$ the Zariski closure of $X_{\mathbb{R}}$ in $\mathbb{P}^N_{\mathbb{C}}$. Then we have:

\begin{equation*}
\partial_a P \subseteq \bigcup_{k = r(X)}^{N} (\tilde{X}_k^*)^*,
\end{equation*}
where $r(X)$ is the minimal integer $k$ such that $(k+1)$-th secant variety of $X$ is of dimension at least $N-1$. In particular, every irreducible components of $\partial_a P$ is a component of $(\tilde{X}^*_k)^*$ for some $k$.

\end{theo}

\begin{proof}
The proof is exactly the same as the one that appears in \cite{RS} (proof of Theorem 1.1), except that our Theorem \ref{app11} improves and supersedes their Lemma 3.1.
\end{proof}

\bigskip

In \cite{RS2, BHORS, ORSV}, the study of the algebraic boundary of $X_{\mathbb{R}} \subset \mathbb{R}^N$ is done in connection with spectrahedral properties of $\mathrm{Conv}(X_{\mathbb{R}})$. More precisely, it is shown that if $\mathrm{Conv}(X_{\mathbb{R}})$ is a spectrahedron, then the algebraic boundary of $X$ has a rich geometry similar to that of classical determinantal hypersurfaces. Nevertheless, we are still lacking a precise description of the affine real compact varieties $X_{\mathbb{R}} \subset \mathbb{R}^{N}$ such that $\mathrm{Conv}(X_{\mathbb{R}})$ is spectrahedral. As part of a general work on spectrahedral representations of convex hulls of semi-algebraic set, Helton and Nie \cite{HN} conjectured that any semi-algebraic convex sets in $\mathbb{R}^N$ is a spectrahedral shadow. Scheiderer recently exhibited many counter-examples to this conjecture \cite{Scheiderer}. In particular, Scheiderer proves the following:

\smallskip

\begin{prop}[\cite{Scheiderer}, proof of corollary 4.25]
The closed convex hull of the Veronese embedding $v_d(\mathbb{R}^n) \subset \mathbb{R}^{N+1}$ is not a spectrahedral shadow for $n \geq 4$ and $d \geq 4$. 
\end{prop}

 It is difficult not to see that the numerical conditions that appear in the above statement are exactly the same that ensure that Conjecture \ref{conj-moi} fails for the Veronese embedding $v_d(\mathbb{P}^{n-1}) \subset \mathbb{P}^N$ (see Example \ref{exemain} in the introduction). One could then wonder if the properties of being a spectrahedral shadow and verifying Conjecture \ref{conj-moi} are related. Namely, one can ask:
 
\begin{quest}
 
 Let $X_{\mathbb{R}} \subset \mathbb{R}^N$ be an affine variety, we denote by $X$ the Zariski closure of $X_{\mathbb{R}}$ in $\mathbb{P}^N_{\mathbb{C}}$. Assume that $\mathrm{Conv}(X_{\mathbb{R}})$ is a spectrahedral shadow. Let $r$ be a positive integer, is it true that:
 
 \begin{equation*}
 \dim X^*_r \leq N-r-1,
 \end{equation*}
 where $X^*_r := \{ H^{\perp} \in X^*, \, \textrm{such that} \, \dim \langle (H \cap X)_{tan} \rangle \geq r \}$?
 \end{quest}

A positive answer to this question would provide a quite effective criterion to decide if the convex hull of an affine real variety is a spectrahedral shadow.

\end{subsection}

\end{section}
\newpage

\bibliographystyle{alpha}

\bibliography{bibliconvex}

\newpage

\end{document}